\documentclass[10pt]{amsart}

\usepackage{amsmath,amssymb}
\usepackage{color}
\usepackage{mathrsfs}
\usepackage{mathptmx}
\usepackage{verbatim}
\usepackage{epsfig}
\usepackage{CJK}
\usepackage{bm}
\usepackage{amsfonts}
\usepackage{amsthm}
\input{epsf.sty}
\usepackage{epsfig}

\title[ ]{ANDERSON LOCALIZATION FOR LONG-RANGE
  OPERATORS WITH SINGULAR POTENTIALS }

\author{Wenwen Jian}
\address[Wenwen Jian]{School of Mathematical Sciences,
Fudan University,
Shanghai 200433, China} \email{wwjian16@fudan.edu.cn}

\author{Jia Shi}
\address[Jia Shi]{School of Mathematical Sciences,
Fudan University,
Shanghai 200433, China} \email{15110180007@fudan.edu.cn}

\author{Xiaoping Yuan}
\address[Xiaoping Yuan]{School of Mathematical Sciences,
Fudan University,
Shanghai 200433, China} \email{xpyuan@fudan.edu.cn}

\keywords{Anderson localization, long-range operators, singular potenials.}

\theoremstyle{plain}
\newtheorem{thm}{Theorem}[section]
 
 \newtheorem{lem}[thm]{Lemma}
 \newtheorem{prop}[thm]{Proposition}
 \newtheorem{rem}[thm]{Remark}
 
 \numberwithin{equation}{section}

\begin{document}


\begin{abstract}
In this paper, we use Cartan estimate for meromorphic functions to prove Anderson localization for a
class of long-range  operators with  singular potenials.
\end{abstract}

\maketitle
\section{Introduction and main result}

In this paper, we study quasi-periodic operators. Start with the almost Mathieu operator
\begin{equation}\label{1}
H=H_{\omega,\lambda,\theta}=\lambda \cos 2\pi (\theta+n\omega)\delta_{nn'}+\Delta,
\end{equation}
where $\Delta$ is the lattice Laplacian on $\mathbb{Z}$
\begin{equation*}
\Delta(n,n')=\left\{
               \begin{array}{ll}
                 1, & |n-n'|=1 , \\
                 0, & |n-n'| \neq1,
               \end{array}
             \right.
\end{equation*}
$\lambda>0$ is the coupling, $\theta \in \mathbb{T} =\mathbb{R}/\mathbb{Z}$ is the phase and $\omega \in \mathbb{R} \backslash\mathbb{Q}$ is the frequency.
Jitomirskaya \cite{J} proved that for Diophantine $\omega$ and almost every $\theta$, the almost Mathieu operator $H$ (\ref{1}) satisfies
Anderson localization for $\lambda>2$. Anderson localization means that $H$ has pure point spectrum with exponentially decaying eigenfunctions.
Bourgain and Goldstein \cite{BG} proved Anderson localization for quasi-periodic Schr\"odinger operators
\begin{equation}\label{2}
H=H_{\omega,\lambda,\theta}=\lambda v (\theta+n\omega)\delta_{nn'}+\Delta
\end{equation}
when $\lambda >\lambda_0(v)$, where $v$ is a nonconstant real analytic potential on $\mathbb{T}$.
Their results is non-perturbative, which means that $\lambda_0$ does not depend on $\omega$.
The proof was based on fundamental matrix and Lyapounov exponent.

If the Laplacian $\Delta$ is replaced by a Toeplitz operator
\begin{equation}\label{3}
S_\phi(n,n')=\hat{\phi}(n-n')
\end{equation}
with $\phi$  real analytic, we obtain the long-range  operators
\begin{equation}\label{4}
H=H_{\omega,\lambda,\theta}=\lambda \cos 2\pi (\theta+n\omega)\delta_{nn'}+S_\phi .
\end{equation}
Bourgain and Jitomirskaya \cite{BJ} proved that there exists $\lambda_0(\phi)$ such that
for any Diophantine $\omega$ and almost every $\theta$, $H$ (\ref{4}) satisfies
Anderson localization when $\lambda >\lambda_0$.

More generally, we can study the long-range  operators of the form
\begin{equation}\label{5}
H=H_{\omega }(x)=v (x+n\omega)\delta_{nn'}+\epsilon S_\phi ,
\end{equation}
where $v$ is non-constant and real analytic on $\mathbb{T}$.
Using Green's function estimates, Bourgain \cite{B05} proved that when $0<\epsilon<\epsilon_{0}=\epsilon_{0}(v,\phi)$,
 $H_{\omega }(x)$ satisfies Anderson localization for $(x,\omega)\in\mathbb{T}^{2}$ in a set of full measure.
Note that in the long range case, we cannot use the fundamental matrix formalism.

Recently, using more elaborate semi-algebraic arguments, Bourgain and Kachkovskiy \cite{BK}
proved Anderson localization for two interacting quasi-periodic particles.
For the Anderson localization results of   quasi-periodic operators on $\mathbb{Z}^{d}$, we refer to \cite{CD,BGS,B07,JLS}.

In the cases above, the potentials $v$ are bounded functions and the operators $H$ are bounded.
Now we consider quasi-periodic operators with unbounded potentials.
For example, we can study the Maryland model
\begin{equation}\label{6}
H=H_{\omega,\lambda,\theta}=\lambda \tan \pi(\theta+n\omega)\delta_{nn'}+\Delta ,
\end{equation}
where
\begin{equation}\label{7}
\theta+n\omega-\frac{1}{2} \notin \mathbb{Z},\quad \forall n \in \mathbb{Z} .
\end{equation}
For Diophantine frequencies $\omega$, the Maryland model (\ref{6}) satisfies
Anderson localization for all $\theta$ \cite{S}.
Jitomirskaya and Liu \cite{JL} proved
arithmetic spectral transitions for the Maryland model.
Recently, using transfer matrix and Lyapounov exponent, Jitomirskaya and Yang \cite{JY2} developed a constructive method to prove Anderson localization for the Maryland model.
We can also prove Anderson localization for the Maryland model with long range interactions \cite{SY}.

Now we replace $\tan$ by more general singular potentials.
Kachkovskiy \cite{K} proved Anderson localization for the following class of quasi-periodic Schr\"odinger operators
\begin{equation*}
  H(x)=f (x+n\omega)\delta_{nn'}+\Delta, x\in\mathbb{R} \backslash(\mathbb{Z}+\omega\mathbb{Z}),
\end{equation*}
for Diophantine $\omega$ and almost everywhere $x$, where $f$ is defined on $\mathbb{R} \backslash\mathbb{Z}$, 1-periodic, continuous on $(0,1)$,
$f(0+)=-\infty, f(1-)=+\infty, \log|f|\in L^{1}(0,1),$ and Lipschitz monotone. That is,
there exists $ \gamma>0$ such that $f(y)-f(x)\geq\gamma(y-x)$ for all $0<x<y<1$.
Jitomirskaya and Yang \cite{JY1} study the singular continuous spectrum for  operators
of the form
\begin{equation*}
  H_{\omega,\theta}=\frac{g(\theta+n\omega)}{f(\theta+n\omega)}\delta_{nn'}+\Delta ,
\end{equation*}
where $f$ is an analytic function and $g$ is Lipschitz.

In this paper, we will consider the following class of long-range operators with singular potentials
\begin{equation}\label{8}
H_{\omega }(x)=\frac{g(x+n\omega)}{f(x+n\omega)}\delta_{nn'}+\epsilon S_\phi ,
\end{equation}
where $f,g$ are real analytic on $\mathbb{T}$. This extends the Maryland model.

We will prove the following result:
\begin{thm}
Consider the following long-range operators with singular potentials
\begin{equation}\label{9}
H_{\omega }(x)=v (x+n\omega)\delta_{nn'}+\epsilon S_\phi ,
\end{equation}
where $v=\frac{g}{f}$, $f,g$ are real analytic on $\mathbb{T}$, $f,v$ are nonconstant,
$Z(f)=\{x\in\mathbb{T}:f(x)=0\}\neq\emptyset$. We will always assume
\begin{equation}\label{10}
x+n\omega  \notin Z(f),\quad \forall n \in \mathbb{Z}.
\end{equation}

Assume $\omega \in DC$ (diophantine condition),
\begin{equation}\label{11}
\|k\omega\|>a| k |^{-A},\quad \forall k\in\mathbb{Z}\setminus\{0\}
\end{equation}
and $\phi$  real analytic satisfying
\begin{equation}\label{12}
  |\hat{\phi}(n)|<e^{-\rho|n|},\quad \forall n \in \mathbb{Z}
\end{equation}
for some $\rho>0$. Fix $x_0\in\mathbb{T}$. Then there is $\epsilon_{0}=\epsilon_{0}(f,g,\phi)>0$, such that if
$0<\epsilon<\epsilon_{0}$, for almost all $\omega \in DC$, $H_{\omega}(x_0)$ satisfies Anderson localization.
\end{thm}

Our result is non-perturbative, since $\epsilon_{0}$ does not depend on $\omega$.
In the long range case here, the transfer matrix formalism is not applicable.
Our basic strategy is the same as that in \cite{B05}, which
is based on a combination of large deviation estimates and semi-algebraic set theory.
The key point is the Green's function estimates for
\begin{equation}\label{13}
G_{[0,N]}(x,E)=(R_{[0,N]}(H(x)-E)R_{[0,N]})^{-1},
\end{equation}
where $R_{\Lambda}$ is the restriction operator to $\Lambda\subset\mathbb{Z}$.
The main difficulty here is that $v$ is singular.
We will first prove   Cartan estimate for meromorphic functions in Section 2,
then we can obtain Lojasiewicz type lemma in Section 3, which is needed for
Green's function estimates in Section 4. Finally, we recall some facts about semi-algebraic sets
 and give the proof of Anderson localization in Section 5.

We will use the following notations. For positive numbers $a,b,a\lesssim b$ means $Ca\leq b$ for some constant $C>0$.
$a\ll b$ means $C$ is large. $a\sim b$ means $a\lesssim b$ and $b\lesssim a$.
For $x\in\mathbb{R}$, $\| x\|=\inf\limits_{m\in\mathbb{Z}}|x-m|$.

\section{Cartan estimate for meromorphic functions}

In this section, we will prove Cartan estimate for meromorphic functions.
We need the following lower bounds for analytic functions.
\begin{thm}[Theorem 11.2 in \cite{L}]\label{t2.1}
  If an analytic function $f(z)$ has no zeros in a disk $\{z\in\mathbb{C}:|z|\leq R\}$
and if $|f(0)|=1$, then
\begin{equation*}
  \log|f(z)|\geq -\frac{2r}{R-r}\log M_f(R), \forall |z|=r<R,
\end{equation*}
where $M_f(R)=\max\limits_{|z|=R}|f(z)|$.
\end{thm}

\begin{thm}\label{t2.2}
  Take $0<R_2<R_1=R\leq 1$,
  $f(z)$ is a meromorphic function in the disk
$\{z\in\mathbb{C}:|z|\leq R\}$ with neither zeros nor poles on $\{z\in\mathbb{C}:|z|= R\}$
and  $|f(0)|=1$.
Let $a_{1},\ldots,a_{n}$ be the zeros of $f(z)$ in $\{z\in\mathbb{C}:|z|< R\}$
and $b_{1},\ldots,b_{n^{\prime}}$ be the poles of $f(z)$ in $\{z\in\mathbb{C}:|z|< R\}$,
where we write down each zero and pole as many times as its multiplicity.
Assume $|b_m|\geq\delta>0, 1\leq m\leq n^{\prime}$.
Given $0<H<1, 0<H^{\prime}<1 $, then
there exists a system of disks $C_{j},C^{\prime}_{j^{\prime}},1\leq j \leq n, 1\leq j^{\prime}\leq n^{\prime}$
 with radii $r_{j}=H,r^{\prime}_{j^{\prime}}=H^{\prime}$
such that the estimate
\begin{equation*}
  \log|f(z)|\geq -\frac{2R_{2}}{R-R_{2}}\log M_f(R)-n\log\frac{R+R_{2}}{H}
-n^{\prime}\left[\frac{2R}{R-R_{2}}\log\frac{1}{\delta}+\log\frac{R(R+R_{2})}{H^{\prime}}\right]
 \end{equation*}
is valid in   $\{z\in\mathbb{C}:|z|\leq R_{2}\}$ outside these disks.
\end{thm}

\begin{proof}
The proof of Cartan estimate for analytic functions can be found in \cite{L}.
Following the same idea as in \cite{L} with minor modification, we give the  proof for  meromorphic  functions.

Let
\begin{equation}\label{2.1}
\psi(z)=f(z)\left[\prod_{|b_m|<R}\frac{R(z-b_m)}{R^2-\overline{b_{m}} z}\right]\left[\prod_{|a_m|<R}\frac{R(z-a_m)}{R^2-\overline{a_{m}} z}\right]^{-1}
\left(\prod_{|b_m|<R}\frac{R}{b_m}\right)\left(\prod_{|a_m|<R}\frac{a_m}{R}\right),
\end{equation}
then $\psi(z)$ is an analytic function without zeros in $\{z\in\mathbb{C}:|z|\leq R\}$ and $|\psi(0)|=1$.
By Theorem \ref{t2.1}, we have
\begin{equation}\label{2.2}
  \log|\psi(z)|\geq -\frac{2R_{2}}{R-R_{2}}\log M_\psi(R), \forall |z|\leq R_{2}.
\end{equation}

Hence
\begin{equation}\label{2.3}
\log|f(z)|\geq -\frac{2R_{2}}{R-R_{2}}\log M_\psi(R)+\log\prod_{|a_m|<R}\frac{R|z-a_m|}{|R^2-\overline{a_{m}} z|}
-\log\prod_{|b_m|<R}\frac{R|z-b_m|}{|R^2-\overline{b_{m}} z|}
\end{equation}
\begin{equation*}
-\log\prod_{|b_m|<R}\frac{R}{|b_m|}-\log \prod_{|a_m|<R}\frac{|a_m|}{R},   \forall |z|\leq R_{2}.
\end{equation*}

Let
\begin{equation}\label{2.4}
s_1=\log\prod_{|a_m|<R}\frac{R|z-a_m|}{|R^2-\overline{a_{m}} z|}=n\log R +\log\prod_{|a_m|<R}|z-a_m|-\log\prod_{|a_m|<R}|R^2-\overline{a_{m}} z|.
\end{equation}

Since $\prod\limits_{|a_m|<R}|z-a_m|\geq H^n$ outside disks
$C_{j} ,1\leq j \leq n $
 with radii $r_{j}=H,$ we have
\begin{equation}\label{2.5}
s_1\geq n\log R +n\log H-n\log R(R+R_{2})=-n\log\frac{R+R_{2}}{H}
\end{equation}
holds in  $\{z\in\mathbb{C}:|z|\leq R_{2}\}$ outside   disks
$C_{j} ,1\leq j \leq n $
 with radii $r_{j}=H$.

Let
\begin{equation}\label{2.6}
s_2=-\log\prod_{|b_m|<R}\frac{R|z-b_m|}{|R^2-\overline{b_{m}} z|}
=-n^{\prime}\log R -\log\prod_{|b_m|<R}|z-b_m|+\log\prod_{|b_m|<R}|b_m|+\log\prod_{|b_m|<R}|z-\frac{R^{2}}{\overline{b_{m}}} |.
\end{equation}

Since $\prod\limits_{|b_m|<R}|z-\frac{R^{2}}{\overline{b_{m}}}|\geq (H^{\prime})^{n^{\prime}}$ outside disks
$C^{\prime}_{j^{\prime}},  1\leq j^{\prime}\leq n^{\prime} $
 with radii $r^{\prime}_{j^{\prime}}=H^{\prime},$ we have
\begin{equation}\label{2.7}
s_2\geq -n^{\prime}\log R -n^{\prime}\log (R+R_{2})+n^{\prime}\log \delta+n^{\prime}\log H^{\prime}=-n^{\prime}\left[\log\frac{1}{\delta}+\log\frac{R(R+R_{2})}{H^{\prime}}\right]
\end{equation}
holds in  $\{z\in\mathbb{C}:|z|\leq R_{2}\}$ outside   disks
$C^{\prime}_{j^{\prime}},  1\leq j^{\prime}\leq n^{\prime} $
 with radii $r^{\prime}_{j^{\prime}}=H^{\prime}$.

By (\ref{2.1}),
\begin{equation}\label{2.8}
M_\psi(R)=M_f(R) \left(\prod_{|b_m|<R}\frac{R}{|b_m|}\right)\left(\prod_{|a_m|<R}\frac{|a_m|}{R}\right).
\end{equation}

Using (\ref{2.3}), (\ref{2.5}), (\ref{2.7}), (\ref{2.8}),
\begin{equation*}
\log|f(z)|\geq -\frac{2R_{2}}{R-R_{2}}\log M_f(R)-\frac{R+R_{2}}{R-R_{2}}\log\prod_{|b_m|<R}\frac{R}{|b_m|}
\end{equation*}
\begin{equation*}
  -\frac{R+R_{2}}{R-R_{2}}\log \prod_{|a_m|<R}\frac{|a_m|}{R}
  -n\log\frac{R+R_{2}}{H}-n^{\prime}\left[\log\frac{1}{\delta}+\log\frac{R(R+R_{2})}{H^{\prime}}\right]
\end{equation*}
\begin{equation*}
  \geq -\frac{2R_{2}}{R-R_{2}}\log M_f(R)-n\log\frac{R+R_{2}}{H}
-n^{\prime}\left[\frac{2R}{R-R_{2}}\log\frac{1}{\delta}+\log\frac{R(R+R_{2})}{H^{\prime}}\right]
 \end{equation*}
holds in  $\{z\in\mathbb{C}:|z|\leq R_{2}\}$ outside   disks
$C_{j}  $, $C^{\prime}_{j^{\prime}}, 1\leq j \leq n, 1\leq j^{\prime}\leq n^{\prime} $
 with radii $r_{j}=H$, $r^{\prime}_{j^{\prime}}=H^{\prime}$.
\end{proof}

\section{Lojasiewicz type lemmas}

For real analytic functions, we have the following
Lojasiewicz type inequality.

\begin{lem}[Lemma 7.3 in \cite{B05}]\label{l3.1}
Let $v_0$ be non-constant real analytic on $\mathbb{T}$,
there is a constant $c_0=c_0(v_0)>0$ such that
\begin{equation*}
 \mathrm{mes}\{x\in\mathbb{T}:|v_0(x)-E|<\delta\}<\delta^{c_0}
\end{equation*}
for all $E\in\mathbb{R}$ and sufficiently small $\delta>0$.
\end{lem}

Using the Cartan estimate in Section 2, we can prove
Lojasiewicz type inequality
for meromorphic functions.

\begin{lem}\label{l3.2}
Let $v=\frac{g}{f}$, $f,g$ are real analytic on $\mathbb{T}$, $f,v$ are nonconstant,
$Z(f)=\{x\in\mathbb{T}:f(x)=0\}\neq\emptyset$.
There is a constant $c_0=c_0(v)>0$ such that
\begin{equation*}
 \mathrm{mes}\{x\in\mathbb{T}:|v(x)-E|<\epsilon\}<\epsilon^{c_0}
\end{equation*}
for all $E\in\mathbb{R}$ and sufficiently small $\epsilon>0$.
\end{lem}

\begin{proof}
  Consider a covering $[0,1]\subset\cup_{l=1}^{m}(p_l-r_l,p_l+r_l)$,
where $0<|v'(p_l)|<\infty $. By Theorem \ref{t2.2},
$\{x\in(p_l-r_l,p_l+r_l):|v'(x)|\leq\epsilon^{\frac{1}{2}}\}$
is contained in a union of at most $C$ intervals of total measure at most
$C\epsilon^{c}$, where $C=C(v),c=c(v)$.
Hence
\begin{equation*}
  \mathrm{mes}\{x\in\mathbb{T}:|v(x)-E|<\epsilon\}\leq
\mathrm{mes}\{x\in\mathbb{T}:|v'(x)|\leq\epsilon^{\frac{1}{2}}\}
\end{equation*}
\begin{equation*}
+\mathrm{mes}\{x\in\mathbb{T}:|v(x)-E|<\epsilon,|v'(x)|>\epsilon^{\frac{1}{2}}\}
\leq C\epsilon^{c}+C\epsilon^{\frac{1}{2}}<\epsilon^{c_0}.
\end{equation*}
\end{proof}

By Lemma \ref{l3.1}, Lemma \ref{l3.2}, we have
\begin{lem}\label{l3.3}
Let $v ,  f,g$ as in Lemma \ref{l3.2}, there is a constant $c=c(f,g)>0$ such that
\begin{equation*}
 \mathrm{mes}\{x\in\mathbb{T}:|g(x)-Ef(x)|<\epsilon\}<\epsilon^{c}
\end{equation*}
for all $E\in\mathbb{R}$ and sufficiently small $\epsilon>0$.
\end{lem}
\begin{proof}
\begin{equation*}
  \mathrm{mes}\{x\in\mathbb{T}:|g(x)-Ef(x)|<\epsilon\}\leq
\mathrm{mes}\{x\in\mathbb{T}:|g(x)-Ef(x)|<\epsilon,|f(x)|\geq\epsilon^{\frac{1}{2}}\}
\end{equation*}
\begin{equation*}
  +\mathrm{mes}\{x\in\mathbb{T}:|f(x)|<\epsilon^{\frac{1}{2}}\}\leq
\mathrm{mes}\{x\in\mathbb{T}:|v(x)-E|<\epsilon^{\frac{1}{2}}\}+\epsilon^{c_0}<\epsilon^{c}.
\end{equation*}
\end{proof}

We introduce the uniform Lojasiewicz   inequality.

\begin{lem}[Lemma 6.1 in \cite{DK}]\label{l3.4}
We denote by $C_{r}^{\omega}(\mathbb{T},\mathbb{R})$
the Banach space of real analytic functions with continuous extension to $A_{r}=\{z\in\mathbb{C}:1-r<|z|<1+r\}$
and norm $\|f\|_{r}=\sup\limits_{z\in A_{r}}|f(z)|$.
Let $0\neq f\in C_{r}^{\omega}(\mathbb{T},\mathbb{R}) $.
Then there are constants $\delta=\delta(f)>0, S=S(f)>0, b=b(f)>0$
such that if $g\in C_{r}^{\omega}(\mathbb{T},\mathbb{R}) $ with
$\|g-f\|_{r}<\delta$, then
\begin{equation*}
 \mathrm{mes}\{x\in\mathbb{T}:|g(x)|<t\}<St^{b}, \forall t>0.
\end{equation*}
\end{lem}

Now, we can prove the following lemma.
\begin{lem}\label{l3.5}
Let $   f,g$ as in Lemma \ref{l3.2}, there is a constant $c=c(f,g)>0$ such that
\begin{equation*}
 \mathrm{mes}\left\{x\in\mathbb{T}:\frac{1}{\sqrt{1+E^2}}\big|g(x)-Ef(x)\big|<\epsilon\right\}<\epsilon^{c}
\end{equation*}
for all $E\in\mathbb{R}$ and sufficiently small $\epsilon>0$.
\end{lem}

\begin{proof}
  If $E>0$, consider
\begin{equation}\label{3.1}
s_1=\left\|-\frac{1}{\sqrt{1+E^2}}g+\frac{E}{\sqrt{1+E^2}}f-f\right\|_{r}\leq\frac{\|f\|_{r}+\|g\|_{r}}{\sqrt{1+E^2}}.
\end{equation}
Since $\forall E\geq C_1=C_1(f,g)>0, s_1<\delta(f)$, by Lemma \ref{l3.4},
\begin{equation}\label{3.2}
 \mathrm{mes}\left\{x\in\mathbb{T}:\frac{1}{\sqrt{1+E^2}}\big|g(x)-Ef(x)\big|<\epsilon\right\}<S\epsilon^{b}, \forall E\geq C_1 .
\end{equation}
 If $E<0$, consider
\begin{equation}\label{3.3}
s_2=\left\| \frac{1}{\sqrt{1+E^2}}g-\frac{E}{\sqrt{1+E^2}}f-f\right\|_{r}\leq\frac{\|f\|_{r}+\|g\|_{r}}{\sqrt{1+E^2}}.
\end{equation}
Since $\forall E\leq -C_1 , s_2<\delta(f)$, by Lemma \ref{l3.4},
\begin{equation}\label{3.4}
 \mathrm{mes}\left\{x\in\mathbb{T}:\frac{1}{\sqrt{1+E^2}}\big|g(x)-Ef(x)\big|<\epsilon\right\}<S\epsilon^{b}, \forall E\leq -C_1 .
\end{equation}
This proves Lemma \ref{l3.5} for $|E|\geq C_1$.
If $|E|\leq C_1$, Lemma \ref{l3.5} follows from Lemma \ref{l3.3}.
\end{proof}

\section{Green's function estimates}

In this section, we will prove Green's function estimates
for
\begin{equation}\label{a1}
G_{N}(x,E)=(R_{[0,N)}(H(x)-E)R_{[0,N)})^{-1},
\end{equation}
where
\begin{equation}\label{a2}
H (x)=v (x+n\omega)\delta_{nn'}+\epsilon S_\phi ,
\end{equation}
with the potential $v=\frac{g}{f}$, $f,g$ are real analytic on $\mathbb{T}$, $f,v$ are nonconstant,
$Z(f)=\{x\in\mathbb{T}:f(x)=0\}\neq\emptyset$ and
$\phi$  real analytic satisfying
\begin{equation}\label{a3}
  |\hat{\phi}(n)|<e^{-\rho|n|},\quad \forall n \in \mathbb{Z}
\end{equation}
for some $\rho>0$.
Without loss, we assume $\hat{\phi}(0)=0$ and $\|f\|_{\infty}\leq 1$.

We will follow the method in \cite{B05}, but as mentioned in Section 1, the operator $H$ is unbounded and the energy $E$ is unbounded.
Write
\begin{equation}\label{a4}
H_{N}(x)-E=F_{N}(x,E)B_{N}(x,E),
\end{equation}
where
\begin{equation}\label{a5}
F_{N}(x,E)(n,n^{\prime})=\frac{\sqrt{1+E^2}}{f(x+n\omega)}\delta_{nn'},
\end{equation}
\begin{equation}\label{a6}
B_{N}(x,E)(n,n)=\frac{1}{\sqrt{1+E^2}}(g(x+n\omega)-Ef(x+n\omega)),
\end{equation}
\begin{equation}\label{a7}
B_{N}(x,E)(n,n^{\prime})=\frac{\epsilon}{\sqrt{1+E^2}}f(x+n\omega)\hat{\phi}(n-n'), n\neq n' .
\end{equation}
Then
\begin{equation}\label{a8}
G_{N}(x,E)=B_{N}(x,E)^{-1}F_{N}(x,E)^{-1}.
\end{equation}
We will prove   estimates
for
\begin{equation}\label{a9}
|B_{N}(x,E)^{-1}(n,n^{\prime})|=\frac{|\det B_{n,n'}(x,E)|}{|\det B_{N}(x,E)|}, \quad n,n'\in [0,N),
\end{equation}
where $B_{n,n'}(x,E)$ refers to the ($n,n'$)-minor of $B_{N}(x,E)$.

\begin{lem}\label{l4.1}
Under the assumptions above, for $\omega \in DC $, if
$0<\epsilon<\epsilon_{0}$,
\begin{equation*}
  \frac{1}{N}\int_{\mathbb{T}}\log|\det B_{N}(x,E)|dx>\int_{\mathbb{T}}\log \frac{|g(x)-Ef(x)|}{\sqrt{1+E^2}}dx-(\kappa+\epsilon_{0}^{\frac{1}{4}}),
\end{equation*}
where $\epsilon_{0}=\epsilon_{0}(\phi,f,g)>0,\kappa=\kappa(\phi,f,g)>0$.
\end{lem}

\begin{proof}
Assume $f(z), g(z)$ are analytic in $\{z=x+iy\in\mathbb{C}:|y|\leq\rho_{1}\}$.
Fix $0<\alpha<\frac{\rho_{1}}{4}$, by analyticity, there is $\eta_0>0$ such that
\begin{equation}\label{a10}
\inf_{E\in\mathbb{R}}\sup_{\frac{\alpha}{2}<y_0<\alpha}\inf_{x\in\mathbb{R}}\frac{|g(x\pm iy_0)-Ef(x\pm iy_0)|}{\sqrt{1+E^2}}>\eta_0.
\end{equation}
Take $0<\epsilon<\epsilon_{0}<\eta_0^{2}$ .
For $z=x+iy,|y|\leq\frac{\rho_{1}}{2}$, by Hadamard inequality,
\begin{equation}\label{a11}
|\det B_{N}(z,E)|\leq\prod_{0\leq n<N}\left[\frac{|g(z+n \omega)-Ef(z+n \omega)|}{\sqrt{1+E^2}}+C\epsilon_{0}\right].
\end{equation}

By Denjoy-Koksma type inequality (Lemma 12 in \cite{J}),
\begin{equation}\label{a12}
\frac{1}{N} \log|\det B_{N}(z,E)|\leq \frac{1}{N}\sum_{n=0}^{N-1}
\log\left[\frac{|g(x+iy+n \omega)-Ef(x+iy+n \omega)|}{\sqrt{1+E^2}}+C\epsilon_{0}\right]
\end{equation}
\begin{equation*}
  \leq\int_{0}^{1}\log\left[\frac{|g(x+iy )-Ef(x+iy )|}{\sqrt{1+E^2}}+C\epsilon_{0}\right]dx+N^{-\delta_1}<C,
\end{equation*}
where $\delta_1=\delta_1(\omega)>0$.

Consider the diagonal matrix $D$
\begin{equation}\label{a13}
D_{nn}=\frac{ g(x+iy_{0}+n \omega )-Ef(x+iy_{0}+n \omega ) }{\sqrt{1+E^2}}
\end{equation}
and non-diagonal matrix $S$
\begin{equation}\label{a14}
S_{n,n'}=\frac{   f(x+iy_{0}+n \omega ) \hat{\phi}(n-n')}{\sqrt{1+E^2}}, n\neq n'.
\end{equation}

Since
\begin{equation}\label{a15}
B_{N}(x+iy_{0},E)=D+\epsilon S=(I+\epsilon SD^{-1})D,
\end{equation}
we have
\begin{equation}\label{a16}
\frac{1}{N} \log|\det B_{N}(x+iy_{0},E)|=\frac{1}{N} \log|\det (I+\epsilon SD^{-1})|+\frac{1}{N} \log|\det D|.
\end{equation}

By (\ref{a10}) and  Denjoy-Koksma type inequality, we have
\begin{equation}\label{a17}
\frac{1}{N} \log|\det D|=\frac{1}{N} \sum_{n=0}^{N-1}\log|D_{nn}|
=\frac{1}{N} \sum_{n=0}^{N-1}\log(|D_{nn}| +\epsilon_0)-\frac{1}{N} \sum_{n=0}^{N-1}\log\left( 1+\frac{\epsilon_0}{|D_{nn}|}\right)
\end{equation}
\begin{equation*}
  \geq\frac{1}{N} \sum_{n=0}^{N-1}\log\left[\frac{|g(x+iy_{0}+n \omega )-Ef(x+iy_{0}+n \omega ) | }{\sqrt{1+E^2}} +\epsilon_0\right]-\frac{\epsilon_0}{\eta_0}
\end{equation*}
\begin{equation*}
\geq\int_{0}^{1}\log\left[\frac{|g(x+iy_{0})-Ef(x+iy_{0}) | }{\sqrt{1+E^2}} +\epsilon_0\right]dx-N^{-\delta_1}-\frac{\epsilon_0}{\eta_0}
\end{equation*}
\begin{equation*}
>\int_{0}^{1}\log\frac{|g(x+iy_{0})-Ef(x+iy_{0}) | }{\sqrt{1+E^2}} dx-2\frac{\epsilon_0}{\eta_0}.
\end{equation*}
\begin{equation}\label{a18}
\frac{1}{N} \log|\det (I+\epsilon SD^{-1})|=-\frac{1}{N} \log|\det (I+\epsilon SD^{-1})^{-1}|
\end{equation}
\begin{equation*}
=-\frac{1}{N} \log\left|\det \left[I+\sum_{s\geq 1}(-\epsilon SD^{-1})^{s}\right]\right|
\geq-\frac{1}{N}\log\prod_{n=0}^{N-1}\left[1+\sum_{s\geq 1}\|(\epsilon SD^{-1})^{s}e_n\|\right]>-C\frac{\epsilon_0}{\eta_0}.
\end{equation*}

By (\ref{a16}), (\ref{a17}), (\ref{a18}),
\begin{equation}\label{a19}
\frac{1}{N} \log|\det B_{N}(x+iy_{0},E)|>\int_{0}^{1}\log\frac{|g(x+iy_{0})-Ef(x+iy_{0}) | }{\sqrt{1+E^2}} dx-C\frac{\epsilon_0}{\eta_0},
\forall x\in \mathbb{R}.
\end{equation}

Since
\begin{equation}\label{a20}
u(z)=\frac{1}{N} \log|\det B_{N}(z,E)|
\end{equation}
is subharmonic in $\{z=x+iy\in\mathbb{C}:|y|<\rho_{1}\}$, we have for $y_1=\frac{\rho_{1}}{2}$,
\begin{equation}\label{a21}
\int_{0}^{1}u(x+iy_0)dx\leq\frac{y_1-y_0}{y_1}\int_{0}^{1}u(x)dx+\frac{ y_0}{y_1}\int_{0}^{1}u(x+iy_1)dx.
\end{equation}

By (\ref{a12}), (\ref{a19}), (\ref{a21}),
\begin{equation}\label{a22}
  \int_{0}^{1}u(x)dx>\frac{y_1}{y_1-y_0}\int_{0}^{1}\log\frac{|g(x+iy_{0})-Ef(x+iy_{0}) | }{\sqrt{1+E^2}}dx
  -C(\frac{\alpha}{\rho_{1}}+\frac{\epsilon_0}{\eta_0}).
\end{equation}

By Lemma \ref{l3.5},
\begin{equation}\label{a23}
\inf_{E\in\mathbb{R}}\int_{0}^{1}\log\frac{|g(x)-Ef(x)|}{\sqrt{1+E^2}} dx>-C_{f,g}, C_{f,g}>0.
\end{equation}

By subharmonicity,
\begin{equation}\label{a24}
\int_{0}^{1}\log\frac{|g(x)-Ef(x)|}{\sqrt{1+E^2}} dx
\end{equation}
\begin{equation*}
\leq\frac{1}{2}\int_{0}^{1}\log\frac{|g(x+iy_{0})-Ef(x+iy_{0}) | }{\sqrt{1+E^2}}dx
+\frac{1}{2}\int_{0}^{1}\log\frac{|g(x-iy_{0})-Ef(x-iy_{0}) | }{\sqrt{1+E^2}}dx .
\end{equation*}

By (\ref{a23}), (\ref{a24}),
\begin{equation}\label{a25}
\int_{0}^{1}\log\frac{|g(x+iy_{0})-Ef(x+iy_{0}) | }{\sqrt{1+E^2}}dx>-C, C>0.
\end{equation}

Using (\ref{a22}), (\ref{a25}), we obtain
\begin{equation}\label{a26}
 \int_{0}^{1}u(x)dx>\int_{0}^{1}\log\frac{|g(x+iy_{0})-Ef(x+iy_{0}) | }{\sqrt{1+E^2}}dx
 -\frac{y_0}{y_1-y_0}C-C(\frac{\alpha}{\rho_{1}}+\frac{\epsilon_0}{\eta_0}).
\end{equation}

Replace $y_0$ by $-y_0$, using (\ref{a24}), (\ref{a26}), we have
\begin{equation*}
  \int_{0}^{1}u(x)dx>\int_{0}^{1}\log\frac{|g(x)-Ef(x)|}{\sqrt{1+E^2}} dx-C\left(\frac{\alpha}{\rho_{1}}+\epsilon_0^{\frac{1}{2}}\right).
\end{equation*}
This proves Lemma \ref{l4.1}.
\end{proof}

We also need the following large deviation theorem.
\begin{thm}[Theorem 2.3 in \cite{SY}]\label{t4.2}
Let $u:\mathbb{T}\rightarrow\mathbb{R}$ be periodic with bounded subharmonic extension $\tilde{u}$ to $|\mathrm{Im} z|\leq 1$.
Assume $\omega\in DC$ . Then
\begin{equation*}
\mathrm{mes}\left\{x\in\mathbb{T}:\left|\sum_{0\leq |m|<M}\frac{M-|m|}{M^{2}} u(x+m\omega)-\hat{u}(0)\right|>M^{-\sigma}\right\}
<e^{-cM^\sigma},\quad c>0
\end{equation*}
for some $\sigma=\sigma(\omega)>0$.
\end{thm}

Now we can prove Green's function estimates.
\begin{prop}\label{p4.3}
Under the assumptions of Lemma \ref{l4.1}, we have for $\omega \in DC $,
$0<\epsilon<\epsilon_{0}$,
 there is $\Omega=\Omega_{N}(E)\subset\mathbb{T}$ satisfying
\begin{equation*}
\mathrm{mes}\Omega<e^{- c N^{\sigma}}, \quad c,\sigma >0
\end{equation*}
  such that if $x\notin \Omega$, then for some $|m|<\sqrt{N}$,
we have the Green's function estimate
\begin{equation*}
|G_{[0,N)}(x+m\omega,E)(n,n')|<e^{-c_{0}(|n-n'|-(\kappa+\epsilon_{0}^{\delta })N)}, \quad n,n'\in [0,N),
\end{equation*}
where $c_{0}=c_{0}(\rho)>0, \kappa=\kappa(\phi,f,g)>0, \delta=\delta(f,g)>0$.
\end{prop}

\begin{proof}
Take $\tilde{C}>10(C_{f,g}+1)$, where $C_{f,g}$ is in (\ref{a23}).
The function
\begin{equation}\label{a27}
 u(x)=\frac{1}{N} \log[|\det B_N(x,E)|+\tilde{C}^{-N}]
\end{equation}
admits a bounded subharmonic extension $u(z)$ to $|\mathrm{Im} z|\leq \rho_{1}$.
By Theorem  \ref{t4.2}, for $x\notin \Omega,\mathrm{mes}\Omega<e^{-cN^\sigma}$,
there is $|m|<\sqrt{N}$, such that $u(x+m\omega)\geq \hat{u}(0)-N^{-\sigma}$.
By Lemma \ref{l4.1},
\begin{equation}\label{a28}
|\det B_N(x+m\omega,E)|\geq e^{N\int_{0}^{1}\log\frac{|g(x)-Ef(x)|}{\sqrt{1+E^2}} dx-(\kappa+\epsilon_{0}^{\frac{1}{4}})N-N^{1-\sigma}}.
\end{equation}

We want to obtain an upper bound on $|\det B_{n,n'}(x,E)|$ uniformly in $x$.
Express $\det B_{n,n'}(x,E)$ as a sum over paths $\gamma$ as
\begin{equation}\label{a29}
\sum_{s}\sum_{|\gamma|=s}\pm\det[R_{[0,N)\backslash\gamma}B_{N}(x,E)R_{[0,N)\backslash\gamma}]
\left(\frac{\epsilon}{\sqrt{1+E^2}}\right)^{s-1}
\prod_{i=1}^{s-1}\left[\hat{\phi}(\gamma_{i+1}-\gamma_{i})f(x+\gamma_{i+1}\omega)\right],
\end{equation}
where $\gamma=(\gamma_{1},\ldots,\gamma_{s})$ is a sequence in $[0,N)$ with $\gamma_{1}=n,\gamma_{s}=n'$.

Hence
\begin{equation}\label{a30}
|\det B_{n,n'}(x,E)|<\sum_{s}\sum_{|\gamma|=s}\epsilon^{s-1}e^{-\rho\sum\limits_{i=1}^{s-1}|\gamma_{i+1}-\gamma_{i}|}
|\det[R_{[0,N)\backslash\gamma}B_{N}(x,E)R_{[0,N)\backslash\gamma}]|.
\end{equation}

If we denote $b=\sum\limits_{i=1}^{s-1}|\gamma_{i+1}-\gamma_{i}|\geq|n-n'|$ and use the fact that there are at most $2^{s-1}\binom{b}{s-1} (s,b)$-paths, then
\begin{equation}\label{a31}
|\det B_{n,n'}(x,E)|<\sum_{b\geq|n-n'|}\sum_{s\leq b+1}2^{s-1}\binom{b}{s-1}\epsilon^{s-1}e^{-\rho b}\max_{|\gamma|=s}|\det[R_{[0,N)\backslash\gamma}B_{N}(x,E)R_{[0,N)\backslash\gamma}]|.
\end{equation}

By Hadamard inequality,
\begin{equation}\label{a32}
|\det[R_{[0,N)\backslash\gamma}B_{N}(x,E)R_{[0,N)\backslash\gamma}]|
\leq\prod_{k\in[0,N)\backslash\gamma}\left[\frac{|g(x+k\omega)-Ef(x+k\omega)|}{\sqrt{1+E^2}}+\epsilon_{0}(\|\hat{\phi}\|_{1}+1)\right].
\end{equation}

Let
\begin{equation}\label{a33}
S_{1}=\sum_{k\in[0,N)}\log\left[\frac{|g(x+k\omega)-Ef(x+k\omega)|}{\sqrt{1+E^2}}+\epsilon_{0}(\|\hat{\phi}\|_{1}+1)\right].
\end{equation}
By Denjoy-Koksma type inequality and Lemma \ref{l3.5},
\begin{equation}\label{a34}
S_{1}\leq N\int_{0}^{1}\log\left[\frac{|g(x)-Ef(x)|}{\sqrt{1+E^2}}+\epsilon_{0}(\|\hat{\phi}\|_{1}+1)\right]dx+N^{1-\delta_1}
\end{equation}
\begin{equation*}
  \leq N\int_{0}^{1}\log\frac{|g(x)-Ef(x)|}{\sqrt{1+E^2}}dx+N\epsilon_{0}^{\delta_2},\delta_1=\delta_1(\omega)>0,\delta_2=\delta_2(f,g)>0.
\end{equation*}

Let
\begin{equation}\label{a35}
S_{2}=\sum_{k\in\gamma}\log\left[\frac{|g(x+k\omega)-Ef(x+k\omega)|}{\sqrt{1+E^2}}+\epsilon_{0}(\|\hat{\phi}\|_{1}+1)\right],\quad |\gamma|=s.
\end{equation}
By Lemma \ref{l3.5}, using the method of Lemma 11.29 in \cite{B05}, we can prove that
if $|\gamma|=s>\epsilon_{0}^{\delta_0}N$, the for all $x$,
$S_{2}\geq \frac{3}{4}s\log\epsilon_{0}$, where $\delta_0=\delta_0(f,g)>0$.

By (\ref{a31})-(\ref{a35}),
\begin{equation}\label{a36}
|\det B_{n,n'}(x,E)|<\sum_{b\geq|n-n'|}\sum_{s\leq b+1,s\leq\epsilon_{0}^{\delta_0}N}2^{s-1}\binom{b}{s-1}\epsilon^{s-1}e^{-\rho b}
(\frac{1}{\epsilon_{0}})^{s}
e^{N\int_{0}^{1}\log\frac{|g(x)-Ef(x)|}{\sqrt{1+E^2}} dx+N\epsilon_{0}^{\delta_2}}
\end{equation}
\begin{equation*}
  +\sum_{b\geq|n-n'|}\sum_{s\leq b+1,s>\epsilon_{0}^{\delta_0}N}2^{s-1}\binom{b}{s-1}\epsilon^{s-1}e^{-\rho b}
(\frac{1}{\epsilon_{0}})^{\frac{3}{4}s}
e^{N\int_{0}^{1}\log\frac{|g(x)-Ef(x)|}{\sqrt{1+E^2}} dx+N\epsilon_{0}^{\delta_2}}.
\end{equation*}

Follow the proof of Proposition 3.1 in \cite{SY}, we have
\begin{equation}\label{a37}
|\det B_{n,n'}(x,E)|<e^{N\int_{0}^{1}\log\frac{|g(x)-Ef(x)|}{\sqrt{1+E^2}} dx+N\epsilon_{0}^{\delta_2}+N\epsilon_{0}^{\frac{\delta_0}{2}}-\frac{\rho}{2}|n-n'|}.
\end{equation}

Using (\ref{a28}), (\ref{a37}),
we have for $x\notin \Omega$, there is $|m|<\sqrt{N}$, such that
\begin{equation}\label{a38}
|B_{N}(x+m\omega,E)^{-1}(n,n^{\prime})|<e^{-\frac{\rho}{2}|n-n'|+(\kappa+\epsilon_{0}^{\delta })N}.
\end{equation}
This proves the Green's function estimate.
\end{proof}

\begin{rem}\label{r4.4}
In the proof of Proposition \ref{p4.3}, we only need to assume
\begin{equation*}
  \|k\omega\|>\frac{a}{|k|^{A}},\quad \forall 0<|k|\leq N.
\end{equation*}
\end{rem}

\section{Proof of Anderson localization}

In this section, we give the proof of Anderson localization as in \cite{BG}.

We first recall some basic facts of semi-algebraic sets . Let $$ \mathcal{P}=\{P_1,\ldots,P_s\}\subset\mathbb{R}[X_1,\ldots,X_n]$$
be a family of real polynomials whose degrees are bounded by $d$.
A semi-algebraic set is given by
\begin{equation}\label{b1}
S=\bigcup_{j}\bigcap_{l\in L_{j}}\{\mathbb{R}^{n}: P_ls_{jl}0 \},
\end{equation}
where $L_{j}\subset\{1,\ldots,s\},s_{jl}\in\{\leq,\geq,=\}$ are arbitrary.
We say that $S$ has degree at most $sd$ and its degree is the $\inf$ of $sd$ over all representations as in (\ref{b1}).

We need the following quantitative version of the Tarski-Seidenberg principle.
\begin{prop}[\cite{BPR}]\label{p5.1}
Let $S\subset\mathbb{R}^{n}$ be a semi-algebraic set of degree $B$, then any projection of $S$ is semi-algebraic of degree at most $B^{C}, C=C(n)$.
\end{prop}

Next fact deals with the intersection of a semi-algebraic set of small measure and the orbit of a diophantine shift.
\begin{prop}[Corollary 9.7 in \cite{B05}]\label{p5.2}
Let $S\subset[0,1]^{n}$ be semi-algebraic of degree $B$ and ${\rm mes}_{n}S<\eta$.
Let $\omega\in\mathbb{T}^{n}$ satisfy a $DC$ and
\begin{equation*}
  \log B\ll\log N\ll\log\frac{1}{\eta}.
\end{equation*}
Then for any $x_0\in\mathbb{T}^{n}$,
\begin{equation*}
  \#\{k=1,\ldots,N :x_0+k\omega\in S\}<N^{1-\delta}
\end{equation*}
for some $\delta=\delta(\omega)>0$.
\end{prop}

We will make essential use of the following transversality property.
\begin{lem}[Lemma 9.9 in \cite{B05}]\label{l5.3}
 Let $S\subset[0,1]^{2n}$ be a semi-algebraic set of degree $B$ and ${\rm mes}_{2n}S<\eta, \log B\ll\log\frac{1}{\eta}$.
We denote $(\omega,x)\in[0,1]^{n}\times[0,1]^{n}$ the product variable and $\{e_j:0\leq j\leq n-1\}$ the $\omega$-coordinate vectors.
Fix $\epsilon>\eta^{\frac{1}{2n}}$. Then there is a decomposition $S=S_1\cup S_2$,
 $S_1$ satisfying
 \begin{equation*}
{\rm mes}_{n}({\rm Proj}_\omega S_1)<B^{C}\epsilon
 \end{equation*}
 and $S_2$ satisfying the transversality property
\begin{equation*}
 {\rm mes}_{n}(S_2\cap L)<B^{C}\epsilon^{-1}\eta^{\frac{1}{2n}}
\end{equation*}
for any $n$-dimensional hyperplane $L$ such that $\max\limits_{0\leq j\leq n-1}|{\rm Proj}_L(e_j)|<\frac{\epsilon}{100}$.
\end{lem}

By application of the resolvent identity, we have the following
\begin{lem}\label{l5.4}
Let $I\subset\mathbb{Z}$ be an interval of size $N$ and $\{I_{\alpha}\}$ subintervals of size $M= N^{\tau}, \tau>0$ is small.
Assume $\forall k\in I$, there is some $\alpha$ such that
\begin{equation}\label{b2}
\left[k-\frac{M}{4},k+\frac{M}{4}\right]\cap I\subset I_\alpha
\end{equation}
and $\forall \alpha$,
\begin{equation}\label{b3}
|G_{I_{\alpha}}(n_1,n_2)|<e^{-c_0(|n_1-n_2|-(\kappa+\epsilon_{0}^{\delta })M)}, \quad  n_1,n_2\in I_{\alpha}.
\end{equation}
Then
\begin{equation}\label{b4}
|G_{I}(n_1,n_2)|<2e^{c_0(\kappa+\epsilon_{0}^{\delta })M}, \quad  n_1,n_2\in I ,
\end{equation}
\begin{equation}\label{b5}
|G_{I}(n_1,n_2)|<e^{-\frac{1}{2}c_0|n_1-n_2|}, \quad  n_1,n_2\in I ,|n_1-n_2|>\frac{N}{10}.
\end{equation}
\end{lem}

\begin{proof}
For $m,n\in I$, there is some $\alpha$ such that
\begin{equation}\label{b6}
\left[m-\frac{M}{4},m+\frac{M}{4}\right]\cap I\subset I_\alpha.
\end{equation}
By resolvent identity,
\begin{equation}\label{b7}
|G_{I}(m,n)|\leq e^{c_0(\kappa+\epsilon_{0}^{\delta })M}+\sum_{m_1\in I_{\alpha},m_2\notin I_{\alpha}}|G_{I_{\alpha}}(m,m_1)|e^{-\rho|m_1-m_2|}|G_{I}(m_{2},n)|.
\end{equation}
If $|m_1-m|\leq\frac{M}{8}$, then $|m_1-m_2|\geq\frac{M}{8}$, hence
\begin{equation}\label{b8}
\sum_{|m_1-m|\leq\frac{M}{8},m_2\notin I_{\alpha}}|G_{I_{\alpha}}(m,m_1)|e^{-\rho|m_1-m_2|}<M e^{-\rho\frac{M}{8}}
e^{c_0(\kappa+\epsilon_{0}^{\delta })M}<\frac{1}{4}.
\end{equation}
If $|m_1-m|>\frac{M}{8}$, then
\begin{equation}\label{b9}
\sum_{|m_1-m|>\frac{M}{8},m_2\notin I_{\alpha}}|G_{I_{\alpha}}(m,m_1)|e^{-\rho|m_1-m_2|}
<M e^{-c_0\frac{M}{8}}e^{c_0(\kappa+\epsilon_{0}^{\delta })M}<\frac{1}{4}.
\end{equation}
By (\ref{b7}), (\ref{b8}), (\ref{b9}),
\begin{equation}\label{b10}
\max_{m,n\in I}|G_{I}(m,n)|<e^{c_0(\kappa+\epsilon_{0}^{\delta })M}+\frac{1}{2}\max_{m,n\in I}|G_{I}(m,n)|.
\end{equation}
(\ref{b4}) follows from (\ref{b10}).

Take $m,n\in I, |m-n|>\frac{N}{10}$, assume (\ref{b6}), by resolvent identity,
\begin{equation}\label{b11}
|G_{I}(m,n)|\leq \sum_{n_0\in I_{\alpha},n_1\notin I_{\alpha}}|G_{I_{\alpha}}(m,n_0)|e^{-\rho|n_0-n_1|}|G_{I}(n_1,n)|
\end{equation}
\begin{equation*}
  \leq Me^{c_0(\kappa+\epsilon_{0}^{\delta })M}\sum_{|m-n_1|>\frac{M}{4}}e^{-c_0|m-n_1|}|G_{I}(n_1,n)|
\end{equation*}
\begin{equation*}
 \leq M^{t}N^{t} e^{tc_0(\kappa+\epsilon_{0}^{\delta })M}
e^{-c_0(|m-n_1|+\cdots+|n_{t-1}-n_t|)}|G_{I}(n_t,n)|,
\end{equation*}
where $t\leq 10\frac{N}{M},|m-n_1|>\frac{M}{4},\ldots,|n_{t-1}-n_t|>\frac{M}{4}$.

If $|n-n_t|\leq M$, then by (\ref{b4}), (\ref{b11}),
\begin{equation}\label{b12}
 |G_{I}(m,n)|  \leq M^{t}N^{t}e^{tc_0(\kappa+\epsilon_{0}^{\delta })M}e^{-c_0(|m-n|-M)}2e^{c_0(\kappa+\epsilon_{0}^{\delta })M}
 \leq e^{20c_0(\kappa+\epsilon_{0}^{\delta })N-c_0|m-n|} <e^{-\frac{1}{2}c_0|m-n|}.
\end{equation}

If $t= 10\frac{N}{M}$, then by (\ref{b4}), (\ref{b11}),
\begin{equation}\label{b13}
 |G_{I}(m,n)|  \leq M^{t}N^{t}e^{tc_0(\kappa+\epsilon_{0}^{\delta })M}e^{-tc_0\frac{M}{4}}2e^{c_0(\kappa+\epsilon_{0}^{\delta })M}
\leq e^{40c_0(\kappa+\epsilon_{0}^{\delta })N-\frac{5}{2}c_0N}<e^{-2c_0N}<e^{-c_0|m-n|}.
\end{equation}

(\ref{b5}) follows from (\ref{b12}), (\ref{b13}). This proves Lemma \ref{l5.4}.
\end{proof}

Now we can prove the main result.
\begin{thm}\label{t5.5}
Consider the following long-range operators with singular potentials
\begin{equation}\label{b14}
H_{\omega }(x)=v (x+n\omega)\delta_{nn'}+\epsilon S_\phi ,
\end{equation}
where $v=\frac{g}{f}$, $f,g$ are real analytic on $\mathbb{T}$, $f,v$ are nonconstant,
$Z(f)=\{x\in\mathbb{T}:f(x)=0\}\neq\emptyset$.
Assume $\omega \in DC$ (diophantine condition),
\begin{equation}\label{b15}
\|k\omega\|>a| k |^{-A},\quad \forall k\in\mathbb{Z}\setminus\{0\}
\end{equation}
and $\phi$  real analytic satisfying
\begin{equation}\label{b16}
  |\hat{\phi}(n)|<e^{-\rho|n|},\quad \forall n \in \mathbb{Z}
\end{equation}
for some $\rho>0$. Fix $x_0\in\mathbb{T}$. Then there is $\epsilon_{0}=\epsilon_{0}(f,g,\phi)>0$, such that if
$0<\epsilon<\epsilon_{0}$, for almost all $\omega \in DC$, $H_{\omega}(x_0)$ satisfies Anderson localization.
\end{thm}

\begin{proof}
By Shnol's theorem \cite{H}, to establish Anderson localization, it suffices to show that if $\xi=(\xi_n)_{n\in\mathbb{Z}},E\in\mathbb{R}$ satisfy
\begin{equation}\label{b17}
\xi_0=1, |\xi_n|<C|n|,\quad |n|\rightarrow\infty,
\end{equation}
\begin{equation}\label{b18}
H(x_0)\xi=E\xi,
\end{equation}
then
\begin{equation}\label{b19}
|\xi_n|<e^{-c|n|},\quad |n|\rightarrow\infty.
\end{equation}

By Proposition \ref{p4.3},
there is $\Omega=\Omega_{N}(E)\subset\mathbb{T},\mathrm{mes}\Omega<e^{- cN^{\sigma}}$, such that if $x\notin \Omega$,
then for some $|m|<\sqrt{N}$,
\begin{equation}\label{b20}
|G_{[-N,N]}(x+m\omega,E)(n_1,n_2)|<e^{-c_{0}(|n_1-n_2|-(\kappa+\epsilon_{0}^{\delta })N)}, \quad |n_1|,|n_2|\leq N.
\end{equation}

As in Section 4,
\begin{equation}\label{b21}
 |G_{[-N,N]}(x+m\omega,E)(n_1,n_2)|=\frac{|f(x+m\omega+n_2\omega)|}{\sqrt{1+E^2}}| B_{[-N,N]}(x+m\omega,E)^{-1}(n_1,n_2)|.
 \end{equation}
Truncate power series for $f, g$ in (\ref{b21}), $\Omega$ may be assumed semi-algebraic of degree at most $N^8$.
Let $N_1=N^{C_{1}}$, $C_{1}$ is a sufficiently large constant. Then by Proposition \ref{p5.2},
\begin{equation}\label{b22}
\#\{|j|\leq N_1:x_{0}+j\omega\in\Omega\}<N_1^{1-\delta_{1}}, \quad\delta_{1}>0.
\end{equation}
Using (\ref{b22}), we may find an interval $I\subset[0,N_1]$ of size $N$ such that
\begin{equation}\label{b23}
x_{0}+j\omega\notin\Omega, \quad\forall j\in I \cup(-I).
\end{equation}
Then for some $|m_{j}|<\sqrt{N}$,
\begin{equation}\label{b24}
|G_{[a,b]}(x_{0},E)(n_1,n_2)|<e^{-c_{0}(|n_1-n_2|-(\kappa+\epsilon_{0}^{\delta })N)}, \quad n_1,n_2\in[a,b],
\end{equation}
where $[a,b]=[j+m_j-N,j+m_j+N]$.
 By (\ref{b17}), (\ref{b18}), (\ref{b24}),
\begin{equation}\label{b25}
|\xi_j|\leq C \sum_{n\in[a,b],n'\notin[a,b]}e^{-c_{0}(|j-n|-(\kappa+\epsilon_{0}^{\delta })N)}e^{-\rho|n-n'|}|n'|
\leq C N_1e^{c_{0}(\kappa+\epsilon_{0}^{\delta })N}e^{-\frac{c_0}{2}N}<e^{-\frac{c_0}{3}N}.
\end{equation}
Let $j_0$ be the center of $I$, we have
\begin{equation}\label{b26}
1=\xi_{0}\leq\|G_{[-j_0,j_0]}(x_0,E)\|\|R_{[-j_0,j_0]}H(x_0)R_{\mathbb{Z}\setminus[-j_0,j_0]}\xi\|.
\end{equation}

For $|n|\leq j_0$, by (\ref{b25}),
\begin{equation}\label{b27}
|(R_{[-j_0,j_0]}H(x_0)R_{\mathbb{Z}\setminus[-j_0,j_0]}\xi)_{n}|\leq\sum_{|n'|>j_0}e^{-\rho|n-n'|}|\xi_{n'}|
\end{equation}
\begin{equation*}
  \leq\sum_{j_0<|n'|\leq j_0+\frac{N}{2}}e^{-\rho|n-n'|}e^{-\frac{c_0}{3}N}+C\sum_{|n'|>j_0+\frac{N}{2}}e^{-\rho|n-n'|}|n'|
<Ce^{-\frac{c_0}{3}N}+CN_1e^{-\rho\frac{N}{3}}<e^{-\frac{c_0}{4}N}.
\end{equation*}

By (\ref{b26}), (\ref{b27}),
\begin{equation}\label{b28}
\|G_{[-j_0,j_0]}(x_0,E)\|>e^{\frac{c_0}{5}N},
\end{equation}
hence
\begin{equation}\label{b29}
{\rm dist}(E, {\rm spec} H_{[-j_0,j_0]}(x_0))<e^{-\frac{c_0}{5}N}.
\end{equation}

Denote
\begin{equation}\label{b30}
\mathcal{E}_{\omega}=\bigcup_{|j|\leq N_1}{\rm spec} H_{[-j,j]}(x_0) .
\end{equation}
It follows from (\ref{b29}) that if $x\notin\bigcup\limits_{E'\in\mathcal{E}_{\omega}}\Omega(E')$,
then for some $|m|<\sqrt{N}$,
\begin{equation}\label{b31}
|G_{[-N,N]+m}(x,E)(n_1,n_2)|<e^{-c_{0}(|n_1-n_2|-(\kappa+\epsilon_{0}^{\delta })N)}, \quad n_1,n_2\in[-N,N]+m .
\end{equation}

Let $N_2=N^{C_{2}}$, $C_{2}$ is a sufficiently large constant.
Suppose
\begin{equation}\label{b32}
x_{0}+n\omega\notin\bigcup_{E'\in\mathcal{E}_{\omega}}\Omega(E'), \quad \forall \sqrt{N_2}<|n|<2N_2,
\end{equation}
then by (\ref{b31}), there are $|m_{n}|<\sqrt{N}$ such that
\begin{equation}\label{b33}
|G_{[-N,N]+n+m_{n}}(x_{0},E)(n_1,n_2)|<e^{-c_{0}(|n_1-n_2|-(\kappa+\epsilon_{0}^{\delta })N)}, \quad n_1,n_2\in[-N,N]+n+m_{n} .
\end{equation}

Let $\Lambda=\bigcup\limits_{\sqrt{N_2}<n<2N_2}([-N,N]+n+m_{n})\supset[\sqrt{N_2},2N_2]$. By Lemma \ref{l5.4},
\begin{equation}\label{b34}
|G_{\Lambda}(x_{0},E)(n_1,n_2)|<2e^{c_0(\kappa+\epsilon_{0}^{\delta })N},\quad n_1,n_2\in\Lambda,
\end{equation}
\begin{equation}\label{b35}
|G_{\Lambda}(x_{0},E)(n_1,n_2)|<e^{-\frac{c_{0}}{2}|n_1-n_2|}, \quad n_1,n_2\in\Lambda ,|n_1-n_2|>\frac{N_2}{10}  .
\end{equation}

For $\frac{1}{2} N_2 \leq j \leq N_2$, by (\ref{b34}), (\ref{b35}),
\begin{equation}\label{b36}
|\xi_{j}|\leq C\sum_{n\in\Lambda,n'\notin\Lambda}|G_{\Lambda}(x_{0},E)(j,n)|e^{-\rho|n-n'|}|n'|
\end{equation}
\begin{equation*}
\leq CN_2\sum_{|n-j|>\frac{N_2}{10}}e^{-\frac{c_{0}}{2}|n-j|}+CN_2^{2}
e^{c_0(\kappa+\epsilon_{0}^{\delta })N}e^{-\rho\frac{N_2}{4}}
\leq e^{-\frac{c_0}{50}N_2} \leq e^{-\frac{c_0}{50}j}.
\end{equation*}

Now we need to prove (\ref{b32}). Consider for $|j|\leq N_1 $, the set $S_j\subset\mathbb{T}^{2}\times\mathbb{R}$ of $(\omega,x,E')$ where
\begin{equation}\label{b37}
 \| k\omega\|>a|k|^{-A},\quad \forall 0<|k|\leq N,
\end{equation}
\begin{equation}\label{b38}
x\in\Omega(E'),
\end{equation}
\begin{equation}\label{b39}
E'\in{\rm spec} H_{[-j,j]}(x_0) .
\end{equation}

Let
\begin{equation}\label{b40}
S={\rm Proj}_{\mathbb{T}^{2}}S_{j}.
\end{equation}
Since ${\rm mes}\Omega(E')<e^{-c N^{\sigma}}$,
\begin{equation}\label{b41}
{\rm mes}S<N_1e^{- c N^{\sigma}}<e^{-cN^{\frac{\sigma}{2}}}.
\end{equation}
Since $S_{j}$ is a semi-algebraic set of degree at most $N_1^{10}$, by Proposition \ref{p5.1},
$S$ is a semi-algebraic set of degree at most $N_1^{10C}$.

 Take $n=1, B=N_1^{10C}, \eta=e^{-cN^{\frac{\sigma}{2}}}, \epsilon =N_2^{-\frac{1}{10}}$
in Lemma \ref{l5.3}, we have $S=S_1\cup S_2$,
\begin{equation}\label{b42}
{\rm mes}{\rm Proj}_{\omega}S_1<B^{C}\epsilon<N_1^{C}N_2^{-\frac{1}{10}}<N_2^{-\frac{1}{11}}.
\end{equation}
We study the intersection of $S_{2}$ and sets
\begin{equation}\label{b43}
\{(\omega,x_0+n\omega):\omega\in[0,1]\},\quad \sqrt{N_2}<|n|<2N_2,
\end{equation}
where $x_0+n\omega$ are considered mod 1. (\ref{b43}) lies in the parallel lines
\begin{equation}\label{b44}
L=L_{m}^{(n)}=\left[\omega=\frac{x}{n}\right]-\frac{m+x_{0}}{n}e_{\omega}, \quad |m|<N_2.
\end{equation}
Since $|{\rm Proj}_{L}e_{\omega}|<\frac{\epsilon}{100}$, by Lemma \ref{l5.3},
\begin{equation}\label{b45}
{\rm mes}(S_{2}\cap L)<B^{C}\epsilon^{-1}\eta^{\frac{1}{2}}<N_1^{C}N_2^{\frac{1}{10}}e^{-\frac{1}{2}cN^{\frac{\sigma}{2}}}.
\end{equation}
Summing over $n,m$,
\begin{equation}\label{b46}
{\rm mes}\{\omega\in[0,1]:(\omega,x_0+n\omega)\in S_{2},\exists\sqrt{N_2}<|n|<2N_2\}
<N_2^{2}N_1^{C}N_2^{\frac{1}{10}}e^{-\frac{1}{2}cN^{\frac{\sigma}{2}}}<e^{-cN^{\frac{\sigma}{4}}}.
\end{equation}
From (\ref{b42}), (\ref{b46}), we exclude an $\omega$-set of measure $N_2^{-\frac{1}{11}}+e^{-cN^{\frac{\sigma}{4}}}<N_2^{-\frac{1}{12}}$.
Summing over $|j|\leq N_1$, we get an $\omega$-set $\mathcal{R}_{N}, {\rm mes}\mathcal{R}_{N}<N_2^{-\frac{1}{13}}<N^{-10}$,
such that for $\omega\notin\mathcal{R}_{N}$,
\begin{equation}\label{b47}
|\xi_{j}|< e^{-\frac{c_0}{50}|j|},\quad \forall|j|\in\left[\frac{1}{2}N^{C_2},N^{C_2}\right].
\end{equation}

Let
\begin{equation}\label{b48}
\mathcal{R}=\bigcap_{N_0\geq 1}\bigcup_{N\geq N_0}\mathcal{R}_{N},
\end{equation}
then ${\rm mes}\mathcal{R}=0$. If $\omega\notin\mathcal{R}$, then by (\ref{b48}),
there is $N_0\geq 1$ such that $\omega\notin\mathcal{R}_{N}, \forall N\geq N_0$.
By (\ref{b47}),
\begin{equation}\label{b49}
|\xi_{j}|< e^{-\frac{c_0}{50}|j|},
\quad \forall|j|\in\bigcup_{N\geq N_0}\left[\frac{1}{2}N^{C_2},N^{C_2}\right]=\left[\frac{1}{2}N_{0}^{C_2},\infty\right).
\end{equation}
This proves (\ref{b19})  and Theorem \ref{t5.5}.
\end{proof}

\subsection*{Acknowledgment}
This paper was supported by  National Natural
Science Foundation of China (No. 11790272 and No. 11771093).

\end{document}